\documentclass{aptpub}
\usepackage{enumerate,color,graphicx,stmaryrd,psfrag} 
\usepackage{url,amscd}
\RequirePackage[colorlinks,citecolor=blue,urlcolor=blue]{hyperref}
\usepackage{breakurl}

     {\begin{list}{}%
             {\setlength{\leftmargin}{#1}\setlength{\rightmargin}{#1}}%
             \item[]%
     }
     {\end{list}}

\RequirePackage[colorlinks,citecolor=blue,urlcolor=blue]{hyperref}

\newcommand\sG{{\mathcal G}}

\newcommand\oo{\infty}
\newcommand\sm{\setminus}

\newcommand\sF{{\mathcal F}}

\newcommand\Om{\Omega}
\newcommand\om{\omega}
\newcommand\Ga{\Gamma}

\newcommand\sB{{\mathcal B}}

\newcommand\resp{respectively}

\newcommand\sym{{\,\triangle\,}}

\newcommand\sM{{\mathcal M}}
\newcommand\sH{{\mathcal H}}
\newcommand\sL{{\mathcal L}}

\newcommand{\RR}{\mathbb{R}}     
\newcommand{\NN}{\mathbb{N}}     

\newcommand{\PP}{\mathbb{P}}     

\newcommand\la{\lambda}
\newcommand\s{\sigma}

\newenvironment{letlist}{\begin{list}{\rm(\alph{mycount})}%
   {\usecounter{mycount}\labelwidth=1cm\itemsep 0pt}}{\end{list}}

\newtheorem{rema}[thm]{Remark}

 \renewcommand\email{Email:\ }
\newcommand\urladdr{URL:\ }

\numberwithin{equation}{section}
\numberwithin{thm}{section}
\numberwithin{figure}{section}
\numberwithin{lemma}{section}
\numberwithin{remark}{section}

\newcounter{mycount}

\newcommand{\qeds}{\hfill$\square$}

\authornames{Grimmett, Janson, Norris}
\shorttitle{Influence in product spaces}

\begin{document}
\title{Influence in product spaces}

\authorone[University of Cambridge]{Geoffrey  R.\ Grimmett}
\addressone{Statistical Laboratory, Centre for
Mathematical Sciences, Cambridge University, Wilberforce Road,
Cambridge CB3 0WB, UK, \email{\href{mailto:grg@statslab.cam.ac.uk}{\nolinkurl{g.r.grimmett@statslab.cam.ac.uk}}}, \urladdr{\url{http://www.statslab.cam.ac.uk/~grg/}}}

\authortwo[Uppsala University]{Svante Janson}
\addresstwo{Department of Mathematics, Uppsala University,
Box 480, 751 06 Uppsala, Sweden, \email{\href{mailto:svante.janson@math.uu.se}{\nolinkurl{svante.janson@math.uu.se}}},  \urladdr{\url{http://www.math.uu.se/~svante/}}}

\authorthree[University of Cambridge]{James  R.\ Norris}
\addressthree{Statistical Laboratory, Centre for
Mathematical Sciences, Cambridge University, Wilberforce Road,
Cambridge CB3 0WB, UK, \email{\href{mailto:j.r.norris@statslab.cam.ac.uk}{\nolinkurl{j.r.norris@statslab.cam.ac.uk}}}, \urladdr{\url{http://www.statslab.cam.ac.uk/~james/}}}

\begin{abstract}
The theory of influence and sharp threshold is a key tool
in probability and probabilistic combinatorics, with numerous applications. One significant aspect of the theory
is directed at identifying the level of generality of the product probability space that
accommodates the event under study.
We derive the influence inequality for a completely general product space,
by establishing a relationship to the Lebesgue cube studied by  Bourgain, Kahn, Kalai, Katznelson, and Linial (BKKKL)
in 1992. This resolves one of the assertions of BKKKL. 
Our conclusion is valid also in the setting of the generalized influences of Keller.
\end{abstract}




\section{Introduction}\label{sec:intro}

A coin shows heads with probability $p$.
We flip it $n$ times, and we observe whether or not some specified event $A$ occurs.
In studying the associated probability $P_p(A)$, it is often useful to gain information about the
degrees of influence of the individual coin tosses.  We make this statement more precise as follows.

Let $(X_e: e\in E)$ be independent Bernoulli variables with parameter $p$, where $|E|=n<\oo$.
Let $A \subseteq \Om$ where $\Om=\{0,1\}^E$. For $\om\in\Om$ and $e \in E$,
we define the configurations $\om^e$ and $\om_e$ by
$$
\om^e(f)=\begin{cases}\om(f) &\text{if } f \ne e,\\ 1&\text{if } f = e,\end{cases}\qquad
\om_e(f)=\begin{cases}\om(f) &\text{if } f \ne e,\\ 0&\text{if } f = e.\end{cases}
$$
Thus, the configuration $\om^e$ (\resp, $\om_e$) is
derived from $\om$ by `switching on' (\resp, `switching off') 
the variable indexed by $e$. 
The \emph{influence} of $e\in E$
on the event $A$ is defined by
\begin{equation}\label{eq:infB}
I_A(e) = P_p\bigl(1_A(\om^e) \ne 1_A(\om_e)\bigr),
\end{equation}
where $1_A$ denotes the indicator function of $A$,
and $P_p$ is the appropriate probability measure.
That is, the influence of $e$ is the probability that the occurrence of $A$ depends
on the value of $X_e$. 

A systematic theory of influence seems to have been developed first 
by Kahn, Kalai, and Linial \cite{KKL} in 1988, in response to an issue raised by
Ben-Or and Linial \cite{BOrL}. There was a later development
by Talagrand \cite{Tal94} in 1994.
On the other hand, estimates for influences have
been key to a number of important results in probability and probabilistic combinatorics
that predate these papers, sometimes by many years. Perhaps the most famous
such result is the proof by Kesten \cite{Kes80}
that the critical probability of bond percolation on the square lattice
equals $\frac12$. There are now several known ways of proving this
(see \cite[Chap.\ 5]{G-pgs} and \cite{DT}), but Kesten's first proof of 1980
used a bespoke  theory of influence.

Kahn, Kalai, and Linial
\cite{KKL} introduced an inequality for influences in the case $p=\frac12$,
working thus with uniform measure on the discrete cube $\{0,1\}^n$.
This was extended by Bourgain et al.\ \cite{bkkkl} to an influence inequality
for the continuous cube $[0,1]^n$ endowed with Lebesgue measure.
Using a discretization argument, this implies an influence  inequality
for the Bernoulli case with $p \in (0,1)$. This following formulation
of this inequality is a minor perturbation of that of \cite{bkkkl,KKL},
and is given here in a form suitable for applications (see \cite[Thm 4.29]{G-pgs}).

\begin{thm}\label{gjn0}
There exists a universal constant $c>0$ such that,
for any $p\in(0,1)$,
any finite set $E$, and any event $A\subseteq \{0,1\}^E$ satisfying $P_p(A) \in (0,1)$,
\begin{equation}\label{0.2}
\sum_{e\in E} I_A(e) \ge c P_p(A)(1-P_p(A)) \log(1/m),
\end{equation}
where $m=\max_e I_A(e)$.
\end{thm}

It is immediate that \eqref{0.2} implies the existence of some $e \in E$ with
\begin{equation}\label{0.3}
I_A(e) \ge c' P_p(A)(1-P_p(A)) \frac{\log n}n,
\end{equation}
where $n=|E|$ and $c'>0$ is an absolute constant.

There is a  slightly extended version of inequality \eqref{0.2}
due to Talagrand \cite{Tal94}, which holds under the further
condition that the event in question is  \emph{increasing}.
Since
the set $\{0,1\}$ is ordered, the product space $\{0,1\}^E$ is partially ordered. 
An event $A$ in this space is said to be \emph{increasing}
if, whenever  $\om\in A$, $\om\le\om'$, then $\om'\in A$.
It is proved at \cite[Thm 1.1]{Tal94} that \eqref{0.2} may be replaced by 
\begin{equation}\label{eq:00}
P_p(A)(1-P_p(A)) \le cp(1-p)\log\left[\frac 2{p(1-p)}\right]
\sum_{e\in E} \frac{I_A(e)}{\log[1/(p(1-p)I_A(e))]},
\end{equation}
for an increasing event $A$. 
Using the fact that $I_A(e) \le m:=\max_e I_A(e)$, inequality \eqref{eq:00}
implies that
\begin{equation}\label{0}
\sum_{e\in E} I_A(e) \ge \left(\frac{c^{-1}}{p(1-p)
\log[2/(p(1-p))]}\right)P_p(A)(1-P_p(A))\log(1/m).
\end{equation}
Since $0<p<1$, it follows that
\begin{equation}\label{000}
\sum_{e\in E} I_A(e) \ge c' P_p(A)(1-P_p(A))\log(1/m),
\end{equation}
where $c'>0$ is an absolute constant, in agreement with \eqref{0.2}
(and assuming $A$ is increasing).

The connection between the influences $I_A(e)$ and the probability $P_p(A)$ 
is provided by \emph{Russo's formula},
\begin{equation}\label{eq:Russo}
\frac{d}{dp}P_p(A) = \sum_{e\in E} I_A(e),
\end{equation}
for any increasing event $A$. 
Russo \cite{Ru78}
published his formula in 1978, though versions of this natural equality were known earlier
to Barlow and Proschan \cite[p.\ 210]{BP} and Margulis \cite{Mar}.

Russo's formula \eqref{eq:Russo} may be combined with \eqref{0} or \eqref{000}
to obtain lower bounds for the derivative of $P_p(A)$ for an increasing event $A$.
Numerous applications of this inequality have been found in areas such as percolation and 
random graphs.

Since these three early papers \cite{bkkkl, KKL, Tal94} on influence, 
several strands of theory have been developed.
One is to seek influence theorems for non-product measures, for which we
refer the reader to \cite{GG,GG11}. Another
is towards the question of whether there exists a useful influence inequality for
an event in an arbitrary product space, that is, whether an inequality
of the form \eqref{000} holds with the discrete product space $\{0,1\}^E$
replaced by an arbitrary product probability space. It was asserted in \cite{bkkkl} that the latter
is indeed true, but the explanation was omitted (a natural argument uses the measure-space isomorphism theorem, which normally requires
separability; see Section \ref{ss3}). The purpose of the
current note is to state and prove a general form of  this theorem not requiring separability
(see Theorems \ref{gjn2-} and \ref{gjn2}).

See \cite{KalS} for a review of influence and its ramifications, and also
\cite{GS} and \cite[Sect. 4.5]{G-pgs}.

\section{Statement of results}\label{sec:main}

Let $X=(\Om,\sF,P)$ be a probability space, and let $E$ be a finite set with $|E|=n$. 
We write $X^E = (\Om^E, \sF^E, \PP=P^E)$ for the product space
of $n$ copies of $X$.
For an index $e\in E$ and a vector $\psi\in\Om^{E\sm \{e\}}$, we define
the \emph{fibre}
\begin{align*}
F_\psi &= \{\om\in \Om^E: \om(f)=\psi(f) \text{ for } f \ne e\} \\
&\simeq\{\psi\} \times \Om,
\end{align*}
comprising all $\om \in \Om^E$ which agree with $\psi$ off $e$.

Let $A \in \sF^E$ be an event.
The \emph{influence} of $e$ on $A$ is defined as
\begin{equation}\label{g-def}
I_A(e) = P^{E\sm \{e\}}\bigl(\{\psi\in \Om^{E\sm\{e\}}: 0<P(A \cap F_\psi)<1\}\bigr).
\end{equation} 
For economy of notation, the space $X$ is not listed explicitly in $I_A(e)$.

\begin{rema}\label{rem1}
Bourgain et al.\ \cite{bkkkl} make use of a different definition of influence, which
may be expressed in the current context as
$$
I_A'(e) = P^{E\sm\{e\}}\bigl(1_A \text{\rm\ is not constant on } F_\psi\bigr).
$$
By comparison with \eqref{g-def}, we have that $I_A(e) \le I_A'(e)$.
Therefore, lower bounds for $I_A(e)$ are stronger than their
equivalents for $I'_A(e)$.

An unsatisfactory property of the influence $I_A'(e)$ is that one may have
$I_A'(e) \ne I_{A'}'(e)$ for events $A$, $A'$ that differ by a null set. This
observation provoked the revised definition \eqref{g-def} introduced in \cite{G-pgs}. 
More general notions of influence have been discussed in \cite{Hat, Kell10, Kell12},
to which we return at \eqref{g-def2}.
\end{rema}

Let $\sL$ denotes the Lebesgue probability space comprising
the unit interval $[0,1]$ endowed with the Borel $\s$-field $\sB[0,1]$ and
Lebesgue measure $\la$. Our main result for influences
as defined in \eqref{g-def} is the following. This will
be extended to more general influences in Theorem \ref{gjn2}.

\begin{thm}\label{gjn2-}
Let $|E|<\oo$ and $A \in \sF^E$. There exists a measurable event $B$ in
the Lebesgue product space $\sL^E$
such that $\la^E(B) = \PP(A)$, and
$I_B(e) = I_A(e)$ for $e \in E$.
\end{thm}

It follows that the influences of an arbitrary event in the general product space
satisfy an inequality whenever such an inequality 
holds for a general event in the 
Lebesgue product space.
Since $X$ is not generally a partially ordered set, it would be inappropriate to seek
results restricted to increasing events, and in addition the method of proof will not
necessarily respect an existing partial order.   

We state one corollary of Theorem \ref{gjn2-}, which may be compared with
Theorem \ref{gjn0}. The proof is at the end of Section \ref{sec:pf}.

\begin{thm}\label{gjn}
There exists a universal constant $c>0$ such that,
for any probability space $X=(\Om,\sF,P)$,
any finite set $E$, and any event $A\in \sF^E$ satisfying $\PP(A) \in (0,1)$,
\begin{equation}\label{2}
\sum_{e\in E} I_A(e) \ge c \PP(A)(1-\PP(A)) \log(1/m),
\end{equation}
where $m=\max_e I_A(e)$.
\end{thm}

It is immediate, as at \eqref{0.3},  that there exists $e \in E$ with
\begin{equation}\label{3}
I_A(e) \ge c'\PP(A)(1-\PP(A)) \frac{\log n}n,
\end{equation}
where $n=|E|$ and  $c'>0$ is an absolute constant.
By Remark \ref{rem1}, this is stronger than BKKKL's
\cite[Thm 1]{bkkkl}.

Our principal Theorem \ref{gjn2-} may be extended without substantial extra work
to a more general notion of influence, introduced by Keller \cite{Kell10}.
Let $\sM$ be the set of measurable functions $h:[0,1] \to[0,1]$. For $h \in \sM$, 
the \emph{$h$-influence} of $e \in E$ on the event
$A \in \sF^E$ is defined as
\begin{equation}\label{g-def2}
I_A^h(e) = P^{E\sm \{e\}}\bigl(h(P(A \cap F_\psi))\bigr),
\end{equation}
where $\mu(f)$ denotes the expection of $f$ under the probability measure $\mu$.
Thus $I_A^h(e) = I_A(e)$ when $h$ is the indicator function $1_{(0, 1)}$.
The function $h(x)=x(1-x)$ has been considered in \cite{Hat}, and other functions $h$ in
\cite{Kell10}.

One might define the influence $I_A(e)$ via a conditional expectation rather than the `pointwise'
definitions \eqref{g-def} and \eqref{g-def2}. 
With $\sF^E_e$ the sub-$\s$-field of $\sF^E$ generated by $\{\om(f): f \ne e\}$,
\eqref{g-def2} can be written
$$
I_A^h(e) = P^{E\sm\{e\}}\bigl(h(\PP(A \mid \sF^E_e))\bigr).
$$
However, we retain the notation adopted in the prior literature. 

Our main theorem for $h$-influences is as follows.

\begin{thm}\label{gjn2}
Let $h \in \sM$ and $A \in \sF^E$. There exists a measurable event $B$ in
the Lebesgue product space $\sL^E$
such that $\la^E(B) = \PP(A)$, and
$I_B^h(e) = I_A^h(e)$ for $e \in E$.
\end{thm}

This extends Theorem \ref{gjn2-}, and yields a positive 
answer to a question of Keller \cite[Footnote 2]{Kell10},
asking whether $h$-influence inequalities may be extended from
Lebesgue to general spaces. 
Theorem \ref{gjn2} includes Theorem \ref{gjn2-},
and its proof is presented in Section \ref{sec:pf}.

\section{Discussion}\label{sec:hist}

Rather than include here a full discussion
of influence and sharp threshold, we draw the attention of the reader to three relevant points.

\subsection{Borel or Lebesgue?}
We have made no assumption above about the completeness (or not) of the probability space 
$X^E=(\Om^E,\sF^E,\PP)$. For events $A,B \in \sF^E$ such that $P(A \sym B)=0$,
we have from \eqref{g-def} and Fubini's theorem that $I_A(e)=
I_B(e)$ for $e \in E$. It follows that, when working with
definition \eqref{g-def}, one may use either the product 
$\sigma$-field $\sF^E$ or its completion.

\subsection{Form of inequality} There exists a family of influence inequalities, from which
one may select one according to the situation under study. By Theorems \ref{gjn2-}
and \ref{gjn2}, any inequality that is valid for the Lebesgue space has a parallel
inequality for a general product space. In these two theorems, no assumption is made of 
\emph{monotonicity}
of the event in question, or about its \emph{invariance} under a  group of actions on $\Om^E$. 

\subsection{General probability spaces}\label{ss3}
The probability space of possibly greatest practical value for applications is the 
Lebesgue space $\sL^E$, since many spaces of importance, including the Bernoulli product spaces,
may be derived via mappings on $\sL^E$.
It was implied by Bourgain et al.\  \cite{bkkkl} that influence inequalities 
for an \emph{arbitrary} product space may be derived from those for $\sL^E$.
A natural route to a proof of such
a statement would be to use the measure-space isomorphism theorem
(see, for example, \cite[{\S}40]{Hal}, \cite[App.\ A]{jan10}, or \cite[Thm 4.7]{Pet}). In its usual form,
the last theorem places a restriction of separability on the probability space after removal of atoms, and
this limits its naive application in the current  situation.
The \emph{separable} case is discussed in \cite[Sect.\ 4.5]{G-pgs}.

Some probabilists tend to consider non-separable probability spaces with only limited enthusiasm.
The current note was inspired by a desire to understand the assertion of \cite{bkkkl},
and to resolve a slightly obscure corner of probability theory.

\section{Proof of Theorem \ref{gjn2}}\label{sec:pf}

The proof of Theorem \ref{gjn2} is achieved via the three lemmas that follow.
For probability spaces $X_i=(\Om_i,\sF_i,P_i)$,  a mapping
$\phi:\Om_1 \to \Om_2$ is said to be \emph{measure preserving} (from $X_1$ to $X_2$)
if, for all
$B_2 \in \sF_2$, the inverse image $B_1=\phi^{-1}(B_2)$ is measurable
and satisfies $P_1(B_1) = P_2(B_2)$. 

For a finite set $E$ and a measure preserving mapping $\phi$, 
the function $\Phi=\phi^E$ is the measure preserving mapping from $X_1^E$ to $X_2^E$ given by
$\Phi((x_e:e\in E))=(\phi(x_e):e\in E)$.

\begin{lemma}\label{lem3} 
Let $X_i=(\Om_i,\sF_i,P_i)$, $i=1,2$, be probability spaces, and let
$\phi:\Om_1 \to \Om_2$ be measure preserving. Let $E$ be a finite set, and write $\Phi=\phi^E$ as above. 
If $B_2 \in \sF_2^E$
and $B_1=\Phi^{-1}(B_2)$, then $I_{B_1}^h(e) = I_{B_2}^h(e)$ for all $e\in E$
and $h \in \sM$. 
\end{lemma}

\begin{proof}
Let $e \in E$, $h \in \sM$,  $B_2 \in \sF_2$, and $B_1=\Phi^{-1}(B_2)$. 
For  $\psi\in \Om_i^{E\sm \{e\}}$, let
$F_\psi$ be the fibre
$$
F_\psi = \{\omega\in \Om_i^E: \omega(f) = \psi(f) \mbox{ for } f \ne e \}
\cong \{\psi\}\times \Om_i.
$$

Suppose $\nu \in \Om_1^{E\sm\{e\}}$, $\psi \in \Om_2^{E\sm\{e\}}$ satisfy 
$\phi^{E \sm \{e\}}(\nu)=\psi$.
Since $\phi$ is measure preserving on each component,
\begin{equation}\label{5}
P_1\Bigl(\{\nu\}\times \phi^{-1}(B_2 \cap F_\psi)\Bigr) = 
P_2(B_2 \cap F_\psi).
\end{equation}
Now $\{\nu\}\times \phi^{-1}(B_2 \cap F_\psi) = B_1 \cap F_\nu$, so that,
for $u\in\RR$,
$$
P_1^{E\sm\{e\}}\Big(h(P_1(B_1\cap F_\nu))>u\Bigr) = 
P_2^{E\sm \{e\}}\Big(h(P_2(B_2\cap F_\psi))>u\Bigr).
$$
We integrate over $u\in [0,\oo)$ to obtain the claim.
\qeds\end{proof}

A $\sigma$-field of subsets of a set $\Om$ is called
\emph{countably generated} (or \emph{separable}) if
it is generated by some finite or countably infinite collection of subsets
of $\Om$.  

\begin{lemma}\label{lem1} 
Let $X=(\Om,\sF,P)$, $|E|<\oo$, and let $A \in \sF^E$.
There exists a countably generated sub-$\s$-field $\sG$ of $\sF$ such that $A \in \sG^E$.
\end{lemma}

\begin{proof}
Let $\{\sG_i: i \in I\}$ be the set of all countably generated
sub-$\s$-fields of the $\s$-field $\sF$, and let $\sH$ be the union of $\sG_i^E$ as $i$ ranges over $I$.
It is easy to see that $\sH$ is a $\s$-field.
(To see closure under countable unions: let $A_i\in \sH$ for $i=1,2,\dots$.
Then $A_i \in \sG_{j(i)}^E$ for some $j(i)$. Let $\sG_j$ be
generated by the countable subset $\sB_j$ of $\sF$, and 
let $\sB = \bigcup_i \sB_{j(i)}$. Then $\sB$ is countable, and generates thus some $\sG_k$. Hence,
$A_i \in \sG_{j(i)}^E \subseteq \sG_k^E$ for each $i$, so that 
$\bigcup _i A_i \in \sG_k^E\subseteq \sH$.)
Furthermore, $\sH$ is the smallest $\s$-field
containing every \emph{rectangle} of the form $\prod_{e\in E} F_e$, as the $F_e$ range over $\sF$.
Therefore, $\sH=\sF^E$.

Let $A \in \sF^E$. Since $A \in \sH$, there exists $a\in I$ such that $A \in \sG_a^E$.
\qeds\end{proof}

The remainder of the proof is based upon a
concealed version of the measure-space isomorphism theorem. 
In general terms, this last states that (subject to appropriate) 
assumptions) a measure space may be placed in 
correspondence with the Lebesgue space $\sL$. There
are two forms of the measure-space isomorphism theorem.
\begin{letlist}
\item  There exists an isomorphism between the measure rings  
of the measure space and the Lebesgue space (see, for example, \cite[{\S}40]{Hal}). 
\item  There exists a pointwise bijection between 
certain derived sample spaces (see, for example,  \cite[Thm 4.7]{Pet}). 
\end{letlist}

We will not appeal to any general theorem here, but instead
will construct the required mappings explicitly in
a manner requiring
no special consideration of the existence (or not) of atoms.  
This may be achieved either by repeated decimation of sub-intervals of $[0,1]$
(see, for example,  \cite[Sect.\ 2.2]{Rud}),
or by way of a mapping to the Cantor set. We choose to follow the second route here.
See \cite[App.\ A]{jan10} for a discussion of measure-space isomorphisms. 

For $T \subseteq \RR^d$, we denote the Borel $\s$-field of $T$ by $\sB(T)$.
Let $C$ be the Cantor set of all reals of the form 
$$
\sum_{k=1}^\oo \frac 2{3^{k}} a_k,\qquad
(a_k:k \in \NN)\in\{0,1\}^\NN.
$$
We shall make use of the fact that  $C$ is in one-to-one
correspondence with $\{0,1\}^\NN$. 

\begin{lemma}\label{lem2}
Let $A \in \sF^E$, and let $\sG$ be a countably generated sub-$\s$-field
of $\sF$ such that $A \in \sG^E$ (as in Lemma \ref{lem1}). There exists a probability space 
$Z = (C,\sB(C),\mu)$ comprising the Cantor set $C$ together with
its Borel $\s$-field and a suitable probability measure $\mu$, such that following hold.
\begin{letlist}
\item There exists a measure-preserving mapping
$\psi$ from $X$ to $Z$.
\item There exists $G \in \sB(C^E)$ such that  $A=\Psi^{-1}(G)$, where $\Psi=\psi^E$.
\item There exists a measure-preserving mapping
$\gamma$ from $\sL$ to $Z$.
\end{letlist} 
\end{lemma}

This lemma (together with part of the forthcoming proof
of Theorem \ref{gjn2}) may be summarized in the diagrams
\begin{equation}\label{diagram}
\begin{CD}
X @>\psi>>   Z @<\gamma<<  \sL,
\phantom{>>>} A @<\Psi^{-1}<< G @>\Gamma^{-1}>> B,
\end{CD}
\end{equation}
where $\Gamma=\gamma^E$ and $B=\Gamma^{-1}(G)$.

\begin{proof}

(a) 
The existence of $\sG$ is implied by Lemma \ref{lem1}.
Since $\sG$ is finitely generated, we may find subsets $(B_k: k \in \NN)$ of $\Om$
that generate $\sG$. Define  $\psi: \Om \to C$ by
$$
\psi(x) = \sum_{k=1}^\oo \frac 2{3^k} 1_{B_k}(x),
$$
where $1_B$ is the indicator function of $B$, as usual.

Write $\sG'=\{\psi^{-1}(S): S \in \sB(C)\}$. We claim that $\sG= \sG'$.
Since $B_k \in \sG'$ for all
$k$, we have $\sG \subseteq \sG'$. Conversely, since $\psi$ is a sum of $\sG$-measurable
functions, it is $\sG$-measurable, and hence $\sG' \subseteq \sG$.

Let $\mu$ be the probability measure on $(C,\sB(C))$ induced by $\psi$, 
that is $\mu(S) = P(\psi^{-1}(S))$
for $S \in \sB(C)$.  By definition of $\mu$, $\psi$ is measure-preserving
from $X$ to $Z=(C,\sB(C),\mu)$.

\noindent
(b)  Let $\sH$ be the $\s$-field $\{\Psi^{-1}(S) : S \in \sB(C^E)\}$ on $\Om^E$.
By the above, $\sH=\sG^E$.
Consequently, $A \in \sH$, and hence $A = \Psi^{-1}(G)$ for
some $G \in \sB(C^E)$.

\noindent
(c) 
Define $\kappa: C \to[0,1]$ by $\kappa(c) = \mu(C \cap [0,c])$. 
We may take as inverse the function 
$$
\gamma(y) = \inf\{c: \kappa(c) \ge y\}, \qquad y \in [0,1].
$$
Since $\gamma(y) \le c$ if and only if $y \le \kappa(c)$, we have that
$$
\gamma^{-1}(C \cap [0,c]) = [0,\kappa(c)], \qquad c \in C,
$$
so that
$$
\la\bigl(\gamma^{-1}(C \cap [0,c])\bigr) = \kappa(c) = \mu(C \cap [0,c]).
$$
The set $\{C \cap [0,c] : c \in C\}$ is a $\pi$-system that generates $\sB(C)$, and hence
$\gamma$ is measure-preserving from $\sL$ to $Z$. 
\qeds\end{proof}

\begin{proof}[Proof of Theorem \ref{gjn2}]
Let $h\in\sM$, $A\in\sF^E$. We shall use the notation introduced in
Lemmas \ref{lem1}--\ref{lem2}, and we refer the reader
to the diagram \eqref{diagram}.
By Lemmas \ref{lem3} and \ref{lem2}(a, b), $A$ and $G$ have equal measure and $h$-influences.
Write $\Ga=\gamma^E$, and take $B = \Ga^{-1}(G) \subseteq [0,1]^E$.
Since $\Gamma$ is measure-preserving, 
by Lemma \ref{lem3}, $G$ and $B$ have equal probability and $h$-influences.
\qeds\end{proof}

\begin{proof}[Proof of Theorem \ref{gjn}]
This is an immediate corollary of Theorem \ref{gjn2}, on
applying the corresponding result for the Lebesgue space.
The latter result is implied by the work of BKKKL  \cite{bkkkl}, and is explicit
at \cite[Thm 4.33]{G-pgs} (the factor $2$ present in the last reference is
cosmetic only).
\qeds\end{proof}

\section*{Acknowledgements}

The work of GRG and JRN was supported in part by the EPSRC under grant EP/103372X/1,
and of SJ by the Knut and Alice Wallenberg Foundation.

\bibliographystyle{apt}
\bibliography{gjn5}

\begin{thebibliography}{10}

\bibitem{BP}
{\sc Barlow, R.~N. and Proschan, F.} (1965).
\newblock {\em {Mathematical Theory of Reliability}}.
\newblock Wiley, New York.

\bibitem{BOrL}
{\sc Ben-Or, M. and Linial, N.} (1990).
\newblock Collective coin flipping.
\newblock In {\em {Randomness and Computation}}.
\newblock Academic Press, New York pp.~91--115.

\bibitem{bkkkl}
{\sc Bourgain, J., Kahn, J., Kalai, G., Katznelson, Y. and Linial, N.} (1992).
\newblock The influence of variables in product spaces.
\newblock {\em Israel J. Math.\/} {\bf 77,} 55--64.

\bibitem{DT}
{\sc Duminil-Copin, H. and Tassion, V.} (2015).
\newblock A new proof of the sharpness of the phase transition for {B}ernoulli
  percolation on {$\ZZ^d$}.
\newblock \url{http://arxiv.org/abs/1502.03051}.

\bibitem{GS}
{\sc Garban, C. and Steif, J.~E.} (2015).
\newblock {\em {Noise Sensitivity of {B}oolean Functions and Percolation}}.
\newblock Cambridge University Press, Cambridge.

\bibitem{GG}
{\sc Graham, B.~T. and Grimmett, G.~R.} (2006).
\newblock Influence and sharp threshold theorems for monotonic measures.
\newblock {\em Ann. Probab.\/} {\bf 34,} 1726--1745.

\bibitem{GG11}
{\sc Graham, B.~T. and Grimmett, G.~R.} (2011).
\newblock Sharp thresholds for the random-cluster and {I}sing models.
\newblock {\em Ann. Appl. Probab.\/} {\bf 21,} 240--265.

\bibitem{G-pgs}
{\sc Grimmett, G.~R.} (2010).
\newblock {\em {Probability on Graphs}}.
\newblock Cambridge University Press, Cambridge.
\newblock \url{http://www.statslab.cam.ac.uk/~grg/books/pgs.html}.

\bibitem{Hal}
{\sc Halmos, P.} (1974).
\newblock {\em {Measure Theory}}.
\newblock Springer, Berlin.

\bibitem{Hat}
{\sc Hatami, H.} (2009).
\newblock Decision trees and influence of variables over product probability
  spaces.
\newblock {\em Combin. Probab. Comput.\/} {\bf 18,} 357--369.

\bibitem{jan10}
{\sc Janson, S.} (2013).
\newblock {\em Graphons, cut norm and distance, couplings and rearrangements}
  vol.~4 of {\em New York Journal of Mathematics Monographs}.
\newblock \url{http://nyjm.albany.edu/m/}.

\bibitem{KKL}
{\sc Kahn, J., Kalai, G. and Linial, N.} (1988).
\newblock The influence of variables on {B}oolean functions.
\newblock In {\em {Proceedings of 29th Symposium on the Foundations of Computer
  Science}}.
\newblock Computer {Science Press} pp.~68--80.

\bibitem{KalS}
{\sc Kalai, G. and Safra, S.} (2006).
\newblock Threshold phenomena and influence.
\newblock In {\em {Computational Complexity and Statistical Physics}}. ed.
  A.~G. Percus, G.~Istrate, and C.~Moore.
\newblock Oxford University Press, New York pp.~25--60.

\bibitem{Kell10}
{\sc Keller, N.} (2011).
\newblock On the influences of variables on {B}oolean functions in product
  spaces.
\newblock {\em Combin. Probab. Comput.\/} {\bf 20,} 83--102.

\bibitem{Kell12}
{\sc Keller, N., Mossel, E. and Sen, A.} (2012).
\newblock Geometric influences.
\newblock {\em Ann. Probab.\/} {\bf 40,} 1135--1166.

\bibitem{Kes80}
{\sc Kesten, H.} (1980).
\newblock The critical probability of bond percolation on the square lattice
  equals {${1\over 2}$}.
\newblock {\em Commun. Math. Phys.\/} {\bf 74,} 41--59.

\bibitem{Mar}
{\sc Margulis, G.~A.} (1974).
\newblock Probabilistic characteristics of graphs with large connectivity.
\newblock {\em Problems Info. Trans.\/} {\bf 10,} 174--179.

\bibitem{Pet}
{\sc Petersen, K.} (1983).
\newblock {\em {Ergodic Theory}}.
\newblock Cambridge University Press, Cambridge.

\bibitem{Rud}
{\sc Rudolph, D.~J.} (1990).
\newblock {\em {Fundamentals of Measurable Dynamics}}.
\newblock Clarendon Press, Oxford.

\bibitem{Ru78}
{\sc Russo, L.} (1978).
\newblock A note on percolation.
\newblock {\em Zeit. f\"ur Wahrsch'theorie verw. Geb.\/} {\bf 43,} 39--48.

\bibitem{Tal94}
{\sc Talagrand, M.} (1994).
\newblock {On Russo's approximate zero--one law}.
\newblock {\em Ann. Probab.\/} {\bf 22,} 1576--1587.

\end{thebibliography}

\end{document}